\documentclass[12pt,letterpaper]{amsart}
\usepackage{amsmath,txfonts}
\usepackage{amssymb}
\usepackage{amsxtra}
\usepackage{amsthm, color}
\usepackage{txfonts}
\usepackage{graphicx}
\usepackage{times}
\usepackage{citeref}
\usepackage{tikz}
\usepackage{hyperref}
\usepackage{stmaryrd}
\usepackage[T3,T1]{fontenc}
\usepackage{pgfplots}

\usetikzlibrary{calc}

\usepackage{mathrsfs}
\usepackage{amsfonts}
\usepackage{amssymb}
\usepackage{ifthen}
\usepackage{graphicx}
\nonstopmode \numberwithin{equation}{section}

\newtheorem{thm}{Theorem}[section]
\newtheorem{lem}{Lemma}[section]
\newtheorem{cor}[thm]{Corollary}
\newtheorem{prop}[thm]{Proposition}

\newtheorem{step}{Step}[section]

\theoremstyle{definition}
\newtheorem{mlem}{Main lemma}[section]
\newtheorem{assertion}{Assertion}[section]
\newtheorem{cl}{Claim}[section]
\newtheorem{ca}{Case}[section]
\newtheorem{sca}{Subcase}[section]
\newtheorem{scl}{Subclaim}[section]
\newtheorem{conj}[thm]{Conjecture}
\newtheorem{fact}{Fact}[section]
\newtheorem{defn}[thm]{Definition}
\newtheorem{op}[thm]{Open Problem}

\newtheorem{ques}{Question}[section]
\newtheorem{rem}[thm]{Remark}
\newtheorem{exam}[thm]{Example}

\numberwithin{equation}{section}

\newcounter {own}
\def\theown {\thesection       .\arabic{own}}

\newenvironment{pf}[1][]{%
 \vskip 3mm
 \noindent
 \ifthenelse{\equal{#1}{}}%
  {{\slshape Proof. }}%
  {{\slshape #1.} }%
 }%
{\qed\bigskip}

\newcounter{alphabet}

\newenvironment{Thm}[1][]{\refstepcounter{alphabet}%
\bigskip%
\noindent%
{\bf Theorem \Alph{alphabet}}%
\ifthenelse{\equal{#1}{}}{}{ (#1)}%
{\bf .} \itshape}{\vskip 8pt}


\newenvironment{Lem}[1][]{\refstepcounter{alphabet}%
\bigskip%
\noindent%
{\bf Lemma \Alph{alphabet}}%
{\bf .} \itshape}{\vskip 8pt}

\newcounter{alphabet2}

\def\be{\begin{equation}}
\def\ee{\end{equation}}

\newcommand{\ben}{\begin{enumerate}}
\newcommand{\een}{\end{enumerate}}

\newcommand{\blem}{\begin{lem}}
\newcommand{\elem}{\end{lem}}
\newcommand{\bthm}{\begin{thm}}
\newcommand{\ethm}{\end{thm}}
\newcommand{\bcor}{\begin{cor}}
\newcommand{\ecor}{\end{cor}}
\newcommand{\beg}{\begin{exam}}
\newcommand{\eeg}{\end{exam}}
\newcommand{\begs}{\begin{examples}}
\newcommand{\eegs}{\end{examples}}
\newcommand{\bdefe}{\begin{defn}}
\newcommand{\edefe}{\end{defn}}
\newcommand{\bques}{\begin{ques}}
\newcommand{\eques}{\end{ques}}
\newcommand{\bei}{\begin{itemize}}
\newcommand{\eei}{\end{itemize}}
\newcommand{\bcon}{\begin{conj}}
\newcommand{\econ}{\end{conj}}
\newcommand{\bop}{\begin{op}}
\newcommand{\eop}{\end{op}}

\newcommand{\bas}{\begin{assertion}}
\newcommand{\eas}{\end{assertion}}

\newcommand{\bfa}{\begin{fact}}
\newcommand{\efa}{\end{fact}}

\newcommand{\bca}{\begin{ca}}
\newcommand{\eca}{\end{ca}}

\newcommand{\bst}{\begin{step}}
\newcommand{\est}{\end{step}}

\newcommand{\bsca}{\begin{sca}}
\newcommand{\esca}{\end{sca}}

\newcommand{\bcl}{\begin{cl}}
\newcommand{\ecl}{\end{cl}}

\newcommand{\bmlem}{\begin{mlem}}
\newcommand{\emlem}{\end{mlem}}

\newcommand{\bscl}{\begin{scl}}
\newcommand{\escl}{\end{scl}}

\newcommand{\bcons}{\begin{conjs}}
\newcommand{\econs}{\end{conjs}}

\newcommand{\bprop}{\begin{prop}}
\newcommand{\eprop}{\end{prop}}

\newcommand{\br}{\begin{rem}}
\newcommand{\er}{\end{rem}}
\newcommand{\brs}{\begin{rems}}
\newcommand{\ers}{\end{rems}}
\newcommand{\bo}{\begin{obser}}
\newcommand{\eo}{\end{obser}}
\newcommand{\bos}{\begin{obsers}}
\newcommand{\eos}{\end{obsers}}
\newcommand{\bpf}{\begin{pf}}
\newcommand{\epf}{\end{pf}}
\newcommand{\ba}{\begin{array}}
\newcommand{\ea}{\end{array}}
\newcommand{\beq}{\begin{eqnarray}}
\newcommand{\beqq}{\begin{eqnarray*}}
\newcommand{\eeq}{\end{eqnarray}}
\newcommand{\eeqq}{\end{eqnarray*}}

\newcounter{minutes}
\divide\time by 60
\newcounter{hours}
\multiply\time by 60 \addtocounter{minutes}{-\time}

\begin{document}

\bibliographystyle{amsplain}
\title []
{Conjugate type   properties   of harmonic  $(K,K')$-quasiregular mappings}

\def\thefootnote{}
\footnotetext{ \texttt{\tiny File:~\jobname .tex,
          printed: \number\day-\number\month-\number\year,
          \thehours.\ifnum\theminutes<10{0}\fi\theminutes}
} \makeatletter\def\thefootnote{\@arabic\c@footnote}\makeatother

\author{Shaolin Chen}
 \address{S. L. Chen, College of Mathematics and
Statistics, Hengyang Normal University, Hengyang, Hunan 421008,
People's Republic of China.} \email{mathechen@126.com}

\author{David Kalaj}
\address{D. Kalaj, University of Montenegro, Faculty of Natural Sciences and
Mathematics,
Cetinjski put b.b. 81000 Podgorica, Montenegro.}
\email{ davidk@ucg.ac.me}




\begin{abstract}
The main purpose of this paper is to investigate conjugate type properties  for harmonic $(K,K')$-quasiregular mappings, where $K \geq 1$ and $K' \geq 0$ are constants. We first establish a Riesz type conjugate function theorem for such mappings, which generalizes and refines several existing results. Additionally, we derive an asymptotically sharp constant for a Riesz type theorem pertaining to a specific class of $K$-quasiregular mappings. Furthermore, we obtain Kolmogorov type and Zygmund type theorems for harmonic $(K,K')$-quasiregular mappings.

\end{abstract}

\maketitle \pagestyle{myheadings} \markboth{ S. L. Chen and D. Kalaj}{Conjugate type   properties    of harmonic  $(K,K')$-quasiregular mappings}

\section{Introduction}\label{csw-sec1}

For $a\in\mathbb{C}$ and $r>0$, define the disk $\mathbb{D}(a,r)=\{z\in\mathbb{C}:~|z-a|<r\}$.
In particular, we write $\mathbb{D}_{r}=\mathbb{D}(0,r)$ for the disk centered at the origin with radius $r$,
and $\mathbb{D}=\mathbb{D}_{1}$ for the unit disk. Let $\mathbb{T}=\partial\mathbb{D}$ denote the
unit circle. The real and imaginary parts of a complex number are denoted by
``${\rm Re}$" and ``${\rm Im}$", respectively.
For a continuously differentiable complex-valued function $f$ and a point $z=x+iy\in\mathbb{D}$,
the directional derivative in the direction $\theta\in[0,2\pi]$ is defined as

$$\partial_{\theta}f(z):=\lim_{\rho\rightarrow0^{+}}\frac{f(z+\rho e^{i\theta})-f(z)}{\rho}=f_{z}(z)e^{i\theta}
+f_{\overline{z}}(z)e^{-i\theta},
$$
where $\rho\in(0,1-|z|)$, and the Wirtinger derivatives are given by
 $$f_{z}:=\frac{\partial f}{\partial z}=\frac{1}{2}\left(\frac{\partial f}{\partial x}-i\frac{\partial f}{\partial y}\right)$$ and
 $$f_{\overline{z}}:=\frac{\partial f}{\partial \overline{z}}=\frac{1}{2}\left(\frac{\partial f}{\partial x}+i\frac{\partial f}{\partial y}\right).$$
Then the maximum and minimum of $\partial_{\theta}f(z)$ over $\theta\in[0,2\pi]$ are given by
$$\Lambda_{f}(z):=\max_{\theta\in[0,2\pi]}\big\{|\partial_{\theta}f(z)|\big\}=|f_{z}(z)|+|f_{\overline{z}}(z)|
$$
and
$$\lambda_{f}(z):=\min_{\theta\in[0,2\pi]}\big\{|\partial_{\theta}f(z)|\big\}=\big||f_{z}(z)|-|f_{\overline{z}}(z)|\big|.
$$

A mapping $f:~\Omega\rightarrow\mathbb{C}$ is said to be absolutely
continuous on lines, $ACL$ in brief, in the domain $\Omega$ if for
every closed rectangle $R\subset\Omega$ with sides parallel to the
axes $x$ and $y$, $f$ is absolutely continuous on almost every
horizontal line and almost every vertical line in $R$. Such a
mapping has, of course,  partial derivatives $f_{x}$ and $f_{y}$
a.e. in $\Omega$. Moreover, we say $f\in ACL^{2}$ if $f\in ACL$ and
its partial derivatives are locally $L^{2}$ integrable in $\Omega$.

A   mapping $f$ of  $\Omega$ into
$\mathbb{C}$ is called a {\it $(K,K')$-quasiregular mapping} if

\ben
\item  $f$ is $ACL^{2}$ in  $\Omega$ and $J_{f}\geq0$ a.e. in   $\Omega$, where $J_{f}$ is the  Jacobian of $f$;

\item there are constants $K\geq1$
and $K'\geq0$ such that
$$\Lambda_{f}^{2}\leq KJ_{f}+K'~\mbox{ a.e. in}~\Omega.$$
\een  In particular, if $K'\equiv0$, then a
$(K,K')$-quasiregular mapping is said to be  $K$-quasiregular. Furthermore, we say $f$ is $K$-quasiconformal in $\Omega$ if $f$ is  $K$-quasiregular and homeomorphic in $\Omega$ (see \cite{C-K,C-P-W,C-P,FS,Ni}).


For a sense-preserving harmonic mapping $f$ defined in $\mathbb{D}$,
the Jacobian of $f$ is given by
$$J_{f}=|f_{z}|^{2}-|f_{\overline{z}}|^{2}=\Lambda_{f}\lambda_{f}
$$
and the second complex dilatation of $f$ is given by
$\mu=\overline{f_{\overline{z}}}/f_{z}$.  It is well-known that every harmonic
mapping $f$ defined in a simply connected domain $\Omega$ admits a decomposition $f = h + \overline{g}$, where $h$ and $g$ are analytic;
this decomposition is unique up to an additive constant.
Recall that $f$ is
sense-preserving in $\Omega$ if $J_{f}>0 $ in $\Omega$.
Thus, $f$ is locally univalent and sense-preserving in $\Omega$
if and only if $J_{f}>0$ in $\Omega$; or equivalently if $h'\neq
0$ in $\Omega$ and  $\mu =g'/h'$ has the
property that $|\mu|<1$ in $\Omega$ (see
\cite{Clunie-Small-84,Du,Lewy}). A mapping $f$ in $\Omega$ is called a {\it harmonic $(K,K')$-quasiregular mapping}  if
 $f$ is a harmonic and $(K,K')$-quasiregular mapping in $\Omega$, where $K\geq1$
and $K'\geq0$  are constants. In particular, if $K'=0$, then a harmonic  $(K,K')$-quasiregular mapping in $\Omega$ is called a harmonic $K$-quasiregular
mapping in $\Omega$. Moreover,  we say $f$ is harmonic $K$-quasiconformal in $\Omega$ if $f$ is harmonic  $K$-quasiregular and homeomorphic in $\Omega$.

The main purpose of this paper is to   investigate the behaviour of the conjugate function  of harmonic  $(K,K')$-quasiregular mappings
when its integral means satisfies different growth conditions.

For convenience, we make a notational convention.
Throughout this paper, we use  $\mathscr{A}$ to denote all  analytic functions of $\mathbb{D}$ into $\mathbb{C}$, and
use $\mathscr{H}$ to denote all harmonic mappings of $\mathbb{D}$ into $\mathbb{C}$. Moveover,
 we use the symbol $C$ to denote various positive
constants, whose values may change from one occurrence to another. Also, we use the notation $C=C(a,b,\ldots)$, which means that the constant $C$ depends only on the given parameters $a$, $b$, $\ldots$.

\section{Preliminaries and  main results}\label{csw-sec1-1}

The {\it  Hardy type space}
$\mathbf{H}^{p}_{g}$ $(p\in(0,\infty])$ consists of all those functions
$f$ of $\mathbb{D}$ into $\mathbb{C}$ such that $f$ is measurable, $M_{p}(r,f)$ exists for all $r\in[0,1)$,
$$\|f\|_{p}:=\sup_{r\in[0,1)}M_{p}(r,f)<\infty$$ for
$p\in(0,\infty)$, and $$\|f\|_{\infty}:=\sup_{r\in[0,1)}M_{\infty}(r,f)<\infty$$ for $p=\infty$,
where
$$M_{p}(r,f)=\left(\frac{1}{2\pi}\int_{0}^{2\pi}|f(re^{i\theta})|^{p}d\theta\right)^{\frac{1}{p}}~\mbox{and}~M_{\infty}(r,f)=\sup_{\theta\in[0,2\pi]}|f(re^{i\theta})|.$$
 In particular, we use $\mathbf{h}^{p}=\mathbf{H}^{p}_{g}\cap\mathscr{H}$ and
$H^{p}:=\mathbf{H}^{p}_{g}\cap\mathscr{A}$
to denote the {\it harmonic Hardy space} and the {\it analytic Hardy space}, respectively. If
$f\in\mathbf{h}^{p}$ for some $p\in (1,\infty)$,
then the radial limits
$$f(\zeta)=\lim_{r\rightarrow1^{-}}f(r\zeta)$$ exist for almost every $\zeta\in\mathbb{T}$, and $f\in L^{p}(\mathbb{T})$ (see \cite[Theorems 6.7,  6.13 and 6.39]{ABR}).
  Since $|f|^{p}$ is subharmonic for $p>1$, we see that $M_{p}(r,f)$ is increasing on $r\in[0,1)$, and
 \be\label{fx-1}
 \|f\|_{p}^{p}=\lim_{r\rightarrow1^{-}}M_{p}^{p}(r,f)=\frac{1}{2\pi}\int_{0}^{2\pi}|f(e^{i\theta})|^{p}d\theta.
 \ee



\subsection*{Conjugate type   properties  of harmonic  $(K,K')$-quasiregular mappings}
Let's recall one of the celebrated results on $\mathbf{h}^{p}$ by Riesz.

\begin{Thm}{\rm (M. Riesz)}\label{M-R}
If  $u\in\mathbf{h}^{p}$  for some $p\in(1,\infty)$, then its
harmonic conjugate $v$ is also of class $\mathbf{h}^{p}$, where $v(0)=0$. Furthermore, there is a constant $C(p)$
such that
\be\label{Riez}M_{p}(r,v)\leq C(p) M_{p}(r,u),~r\in[0,1),\ee
for all $u\in\mathbf{h}^{p}$.
\end{Thm}
In \cite{Pi}, Pichorides improved (\ref{Riez}) into the following sharp form
\be\label{Riez-1}\|v\|_{p}\leq\left(\cot\frac{\pi}{2p^{\ast}}\right)\|u\|_{p},\ee
where $p^{\ast}=\max\{p,~p/(p-1)\}$. Later, \eqref{Riez-1} was further improved by Verbitsky \cite{ver} into the following form

\be\label{Riez-2}\frac{1}{\cos\frac{\pi}{2p^{\ast}}}\|v\|_{p}\leq\|f\|_{p}\leq\frac{1}{\sin\frac{\pi}{2p^{\ast}}}\|u\|_{p},\ee
where $f=u+iv$ and $v(0)=0$. We refer interested readers to \cite{AS-2,CH-2023,chenarxiv,duren,FS-1972,FS-2020,graf,verb1,verb2,tams,mel} for more details on this topic.

In 2023,  Liu and Zhu  \cite{aimzhu} generalized Riesz conjugate functions theorem to
 harmonic $K$-quasiregular mappings  with nonvanishing real part in $\mathbb{D}$ (when $1 < p \leq
2$). It is read as follows. 

\begin{Thm}{\rm  (\cite[Theorem  1.1]{aimzhu})}\label{Thm-A}
Suppose that $f=u+iv$ is a  harmonic  $K$-quasiregular mapping in $\mathbb{D}$ with  $v(0)=0$, and $u\geq0$ $($or $u<0$$)$. If $u\in \mathbf{h}^{p}$ for some $p\in(1,2]$, then $v\in \mathbf{h}^{p}.$
Furthermore, there is a constant $C(K,p)$ such that
$$M_{p}(r,v)\leq~C(K,p)M_{p}(r,u).$$
Moreover,   $C(1,p)=\lim_{K\to 1^{+}}C(K,p)=\cot(\pi/(2p^{\ast}))$,  which coincides with the classical analytic case.
\end{Thm}


For $p\in(2,\infty)$, Liu and Zhu established a Riesz conjugate functions theorem for
 harmonic $K$-quasiconformal mappings as follows

\begin{Thm}{\rm  (\cite[Theorem  1.2]{aimzhu})}\label{Thm-B}
Suppose that $f=u+iv$ is a  harmonic $K$-quasiconformal mapping in $\mathbb{D}$ with  $v(0)=0$.
If $u\in \mathbf{h}^{p}$ for some $p\in(2,\infty)$, then $v\in \mathbf{h}^{p}.$
Furthermore, there is a constant $C(K,p)$, depending only on $K$ and $p$, such that
$$M_{p}(r,v)\leq~C(K,p)M_{p}(r,u).$$
\end{Thm}

We remark that the constant $C(K,p)$ in Theorem C is not asymptotically sharp.
Liu and Zhu stated  that finding the asymptotic sharp constant in Theorem  C is  really challenging
 (see \cite[p. 4 and l. 13-16]{aimzhu}).

In the following, using a relatively simple method compared to  \cite{aimzhu},
 we  remove the assumption of ``$u\geq0$ (or $u<0$)" in Theorem B, and replace the assumption of ``harmonic $K$-quasiregular mappings"  in Theorem B by  the weaker one of ``harmonic $(K,K')$-quasiregular mappings", where $K\geq1$ and $K'\geq0$ are constants. Furthermore, we  replace the assumption of ``harmonic $K$-quasiconformal mappings"  in Theorem C by  the weaker one of ``harmonic $(K,K')$-quasiregular mappings". Our result is as follows.

 %


 \begin{thm}\label{thm-0.1}
 Let $f=u+iv$ be a harmonic $(K,K')$-quasiregular mapping in $\mathbb{D}$ with $v(0)=0$, where $K\geq1$ and $K'\geq0$ are constants.
 If $u\in \mathbf{h}^{p}$ for some $p\in(1,\infty)$, then $v\in \mathbf{h}^{p}.$
 Furthermore, there are constants $C_{j}(K,p)$ ($j=1,2$), depending only on $K$ and $p$, such that
$$M_{p}(r,v)\leq\, C_{1}(K,p)M_{p}(r,u)+C_{2}(K,p)\sqrt{K'}.$$
 \end{thm}


  In particular,  we provide an asymptotically sharp constant for
  the case $p\ge 2$ on a certain class of $K$-quasiregular mappings as follows.



\begin{thm}\label{maint}
Suppose that $f=u+iv$ is a harmonic $K$-quasiregular mapping in $\mathbb{D}$ with  $u\in\mathbf{h}^{p}$, where $p\in[2,\infty)$.

\begin{enumerate}
\item[{\rm $(\mathscr{A}_{1})$}]
 Let $p\ge 2$ be an even integer such that
\begin{equation}\label{eqcon}
\cos (p(\arg f(0)-\pi/2))\le 0.\end{equation} Then there are approximately sharp constants $C(K, p)$ and $c(K, p)$ with
$$\lim_{K\rightarrow1^{+}}C(K, p)=\csc\left(\frac{\pi}{2p}\right)~{\rm and}~\lim_{K\rightarrow1^{+}}c(K, p)=\cot\left(\frac{\pi}{2p}\right)$$
such that
\begin{equation}\label{11}\|f\|_{p}\le C(K, p)\|u\|_p\end{equation} and
\begin{equation}\label{22}\|v\|_{p}\le c(K, p)\|u\|_p,\end{equation} respectively.
\end{enumerate}

\begin{enumerate}
\item[{\rm $(\mathscr{A}_{2})$}]
 If $p\ge 2$ is not an even integer, then the same conclusion as in {\rm $(\mathscr{A}_{1})$} we get under the additional assumption that
$$f(\mathbb{D})\subset\,D_p:=\left\{w\in\mathbb{D}:~\arg w\in \left[-\frac{(p-1) \pi }{2 p},\frac{(3 p-1) \pi }{2 p}\right]\right\}.$$
\end{enumerate} 
\end{thm}

In order to formulate the next result, let us define the subharmonic function $\Phi_p$.
For $p\geq2$, let $$\varphi(t)=\left\{
                                  \begin{array}{ll}
                                    -\cos p(\pi/2-|t|), & \hbox{if $\pi/2-\pi/p\le |t|\le \pi/2$,} \\
                                    \max\{|\cos p(\pi/2-t)|,|\cos p(\pi/2+t)|\}, & \hbox{if $|t|<\pi/2-\pi/p$,}
                                  \end{array}
                                \right.$$ and  $\varphi(t)=\varphi(\pi-|t|)$ if $\pi/2\le |t|\le \pi$.

Let \be\label{rf-u-1}\Phi_p(r e^{it})=r^p \varphi(t), \\ \  t\in[-\pi,\pi], \ \ 0\le r<\infty.\ee
Then $\Phi_p$ is subharmonic (see  \cite{verb2}).

\begin{prop}\label{prop-2}
Assume that $f=u+iv$ is a complex-valued harmonic function in $\mathbb{D}$ such that $Q_p(z):=\Phi_p(f(z))$ is a subharmonic function or more general assume that $$\int_{-\pi}^{\pi}Q_p( e^{it})\frac{dt}{2\pi}\ge Q_p(0).$$ Assume also that  $v(0)=0$ and $u\in h^p$ for $p\geq2$. Then \begin{equation}\label{prima}\|f\|_p\le \csc\frac{\pi}{2p}\|u\|_p\end{equation}
and \begin{equation}\label{seconda}\|v\|_p\le \cot\frac{\pi}{2p}\|u\|_p,\end{equation} where $p\geq2$.
In particular, if $f$ is an analytic function in $\mathbb{D}$, then $\Phi_p(f(z))$ is subharmonic and we get  \eqref{prima} and \eqref{seconda}.
\end{prop}

\begin{rem}
The proof of Proposition \ref{prop-2} is similar to  the proof of Theorem \ref{maint}. Hence we omit it here.
The inequalities \eqref{prima} and \eqref{seconda} for analytic functions  are proved by Verbitsky  \cite{ver} using a duality argument which differs from our direct proof.
\end{rem}






Although the harmonic conjugate of an $\mathbf{h}^{1}$ function $u$ need not be in $\mathbf{h}^{1}$, it does belong to $\mathbf{h}^{q}$ for all $q\in(0,1)$
(cf. \cite[p.405]{H-L} and \cite[Theorem  4.2]{duren}). It is read as follows.

\begin{Thm}{\rm (Kolmogorov's theorem)}\label{Kol}
If $u\in\mathbf{h}^{1}$, then its conjugate $v\in\mathbf{h}^{q}$ for all $q\in(0,1)$, where $v(0)=0$.
\end{Thm}

Since $u\in\mathbf{h}^{1}$ is not enough to ensure its conjugate $v\in\mathbf{h}^{1}$, but the stronger
hypothesis $u\in\mathbf{h}^{p}$ for some $p>1$ is sufficient, it is nature to ask for the ``minimal" growth
restriction on $u$ which will imply $v\in\mathbf{h}^{1}$.  Zygmund consider this question and obtained the
following result (cf.  \cite[p.405]{H-L} and \cite[Theorem  4.3]{duren}).

\begin{Thm}{\rm (Zygmund's theorem)}\label{Zyg}
If $u$ is a harmonic function in $\mathbb{D}$ with
$$\int_{0}^{2\pi}|u(re^{i\theta})|\log^{+}|u(re^{i\theta})|d\theta<\infty,$$
then its conjugate $v\in\mathbf{h}^{1}$, where $v(0)=0$ and $\log^{+}|u(re^{i\theta})|=\max\{0,\log|u(re^{i\theta})|\}$. Moreover, $f=u+iv\in\,H^{1}$.
\end{Thm}

For $p\in(0,1)$, Hardy and Littlewood proved the following result (see \cite[Theorem 7]{H-L}).

\begin{Thm}{\rm (Hardy-Littlewood's Theorem)}\label{HL-1931}
Suppose that $f=u+iv$ is an analytic function in $\mathbb{D}$ with $v(0)=0$. 
If  $u\in\mathbf{h}^{p}$ for
some $p\in(0,1)$, then

$$M_{p}(r,v)=O\left(\left(\log\frac{1}{1-r}\right)^{\frac{1}{p}}\right)$$
as $r\rightarrow1^{-}$.
\end{Thm}

In the following, we extend Theorems D, E and F to harmonic $(K,K')$-quasiregular mappings.

\begin{thm}\label{thm-K-L}
Let $f=u+iv$ be a harmonic $(K,K')$-quasiregular mapping in $\mathbb{D}$ with $v(0)=0$, where $K\geq1$ and $K'\geq0$ are constants.

\begin{enumerate}
\item[{\rm $(\mathscr{B}_{1})$}] If $u\in\mathbf{h}^{1}$, then $v\in\mathbf{h}^{q}$ for all $q\in(0,1)$.
\item[{\rm $(\mathscr{B}_{2})$}] If $$\int_{0}^{2\pi}|u(re^{i\theta})|\log^{+}|u(re^{i\theta})|d\theta<\infty,$$ then $v\in\mathbf{h}^{1}$.
\item[{\rm $(\mathscr{B}_{3})$}] If $u\in\mathbf{h}^{p}$ for some $p\in(0,1)$, then
\be\label{jk-1q}M_{p}(r,v)=O\left(\left(\log\frac{1}{1-r}\right)^{\frac{1}{p}}\right)\ee
as $r\rightarrow1^{-}$. In particular, if  $p=1,\,1/2,\,1/3,\,\ldots$, then the estimate of {\rm (\ref{jk-1q})} is sharp.
\end{enumerate}
\end{thm}

It is well known that every analytic function which has an $H^{1}$ derivative is continuous in the closed disk and absolutely continuous
on the boundary. In  particular, $f'\in H^{1}$ implies $f\in H^{\infty}$. This latter result was generalized by Hardy and Littlewood into the following from:

\begin{Thm}{\rm (\cite[Theorem  5.12]{duren})}\label{Thm-I}
If $f'\in H^{p}$  for some $p<1$, then $f\in H^{q}$, where $q=p/(1-p)$. For each value of $p$, the index $q$ is best possible.

\end{Thm}

In the following, we extend Theorem G to harmonic $(K,K')$-quasiregular mapping in $\mathbb{D}$, where $K\geq1$ and $K'\geq0$.

\begin{thm}\label{thm-07}
Let $f=u+iv$ be  a harmonic $(K,K')$-quasiregular mapping in $\mathbb{D}$ with $v(0)=0$, where $K\geq1$ and $K'\geq0$ are constants.
If $|\nabla u|\in \mathbf{H}^{p}_{g}$  for some $p<1$, then $v,~f\in \mathbf{h}^{q}$, where $q=p/(1-p)$. For each value of $p$, the index $q$ is best possible.
\end{thm}

The proofs of  Theorems   \ref{maint}, \ref{thm-0.1}, \ref{thm-K-L} and \ref{thm-07} will be presented in Sec. \ref{csw-sec2}.

\section{Conjugate type   properties  of harmonic  $(K,K')$-quasiregular mappings}\label{csw-sec2}

For $p\in(0,\infty )$, denote by $L^{p}(\mathbb{T})$ the
set of all measurable functions $\psi$ of $\mathbb{T}$ into
$\mathbb{C}$ with
$$\|\psi\|_{L^{p}}=\left(\frac{1}{2\pi}\int_{0}^{2\pi}|\psi(\zeta)|^{p}d\theta\right)^{\frac{1}{p}}<\infty.$$
The following result follows from \cite[Theorem  1.1]{Pav-2013}.

\begin{Thm}\label{Thm-L-G}
 For $f\in\mathscr{A}$, the Littlewood-Paley $\mathscr{G}[f]$-function of $f$ is
defined as follows
$$\mathscr{G}[f](\zeta)=\left(\int_{0}^{1}(1-r)|f'(r\zeta)|^{2}dr\right)^{\frac{1}{2}},~\zeta\in\mathbb{T}.$$ Then,
for  $p\in(0,\infty)$, there is a positive constant $C=C(p)$, depending only on $p$, such that
$$\frac{1}{C}\|f\|_{p}\leq\|\mathscr{G}[f]\|_{L^{p}}\leq\,C\|f\|_{p}.$$
\end{Thm}

 The following result is well-known.

\begin{Lem}\label{Lemx}
Suppose that $a,~b\in[0,\infty)$ and $p\in(0,\infty)$. Then
$$(a+b)^{p}\leq2^{\max\{p-1,0\}}(a^{p}+b^{p}).$$
\end{Lem}

\subsection*{The proof of Theorem \ref{thm-0.1}}
 Since $\mathbb{D}$ is a simply connected domain, we see that $f$ admits a decomposition
 $f=h + \overline{g}$, where $h$ and $g$ are analytic
 in $\mathbb{D}$ with $g(0)=0$. Let $F_{1}=h+g$ and $F_{2}=h-g$. Then $u={\rm Re}(F_{1})$ and $v={\rm Im}(F_{2})$.
 Since $f$ is a harmonic $(K,K')$-quasiregular mapping in $\mathbb{D}$, we see that
$$(\Lambda_{f}(z))^{2}\leq KJ_{f}(z)+K' ~\mbox{ for $z\in\mathbb{D}$}.
$$
This gives that, for $z\in\mathbb{D}$,
\beqq
\Lambda_{f}(z)\leq\frac{K\lambda_{f}(z)+\sqrt{\big(K\lambda_{f}(z))\big)^{2}+4K'}}{2}
\leq K\lambda_{f}(z)+\sqrt{K'},
\eeqq
and consequently
\be\label{K-1}|g'(z)|\leq\mu_{1}|h'(z)|+\mu_{2},
\ee where $\mu_{1}=(K-1)/(K+1)$ and $\mu_{2}=\sqrt{K'}/(1+K)$.
By (\ref{K-1}), we have

\beqq
|F_{2}'|\leq\Lambda_{f}\leq(1+\mu_{1})|h'|+\mu_{2}
\eeqq
and
\beqq
|F_{1}'|\geq\lambda_{f}\geq(1-\mu_{1})|h'|-\mu_{2},
\eeqq
which imply that
\be\label{eq-ghk-2}
|F_{2}'|\leq\frac{1+\mu_{1}}{1-\mu_{1}}|F_{1}'|+\frac{2\mu_{2}}{1-\mu_{1}}=K|F_{1}'|+\sqrt{K'}.
\ee
 It follows from (\ref{Riez-2}) and the assumption ``$u\in \mathbf{h}^{p}$ for some $p\in(1,\infty)$" that
 \beqq\label{eqry-3}
 \frac{1}{\cos\frac{\pi}{2p^{\ast}}}\|{\rm Im}(F_{1})\|_{p}\leq\|F_{1}\|_{p}\leq\frac{1}{\sin\frac{\pi}{2p^{\ast}}}\|u\|_{p},
 \eeqq
 which, together  with Theorem H, gives that

\be\label{eq-x-1}
 \|\mathscr{G}[F_{1}]\|_{L^{p}}\leq\,C\|F_{1}\|_{p}\leq\frac{C}{\sin\frac{\pi}{2p^{\ast}}}\|u\|_{p},
 \ee
 where $C=C(p)$ is a positive constant. By (\ref{eq-ghk-2}), we see that

 \beqq
 \|\mathscr{G}[F_{2}]\|_{L^{p}}&=&\left(\frac{1}{2\pi}\int_{0}^{2\pi}\left(\int_{0}^{1}(1-r)|F_{2}'(re^{i\theta})|^{2}dr\right)^{\frac{p}{2}}d\theta\right)^{\frac{1}{p}}\\
 &\leq&\left(\frac{1}{2\pi}\int_{0}^{2\pi}\left(\int_{0}^{1}(1-r)\left(2K^{2}|F_{1}'(re^{i\theta})|^{2}+2K'\right)dr\right)^{\frac{p}{2}}d\theta\right)^{\frac{1}{p}},
 \eeqq
 which, together with (\ref{eq-x-1}) and Lemma I, yields that
 \beq\label{x-2}
 \|\mathscr{G}[F_{2}]\|_{L^{p}}&\leq&2^{\frac{\max\{0,\frac{p}{2}-1\}}{p}+\frac{1}{2}}\left( K\|\mathscr{G}[F_{1}]\|_{L^{p}}+\sqrt{K'}\right)\\ \nonumber
 &\leq&2^{\frac{\max\{0,\frac{p}{2}-1\}}{p}+\frac{1}{2}}\left( KC\|F_{1}\|_{p}+\sqrt{K'}\right),
 \eeq
 where $C=C(p)$ is a positive constant.
 From (\ref{eqry-3}), (\ref{x-2}) and Theorem H, there is a positive constant $C=C(p)$ such that

 \beq\label{x-3}
 \|v\|_{p}&\leq&\|F_{2}\|_{p}\leq\,C\|\mathscr{G}[F_{2}]\|_{L^{p}}\\ \nonumber
 &\leq&2^{\frac{\max\{0,\frac{p}{2}-1\}}{p}+\frac{1}{2}}C\left( KC\|F_{1}\|_{p}+\sqrt{K'}\right)\\ \nonumber
 &\leq&2^{\frac{\max\{0,\frac{p}{2}-1\}}{p}+\frac{1}{2}}C\left( \frac{KC}{\sin\frac{\pi}{2p^{\ast}}}\|u\|_{p}+\sqrt{K'}\right).
 \eeq
 For any fixed $r\in[0,1)$, let $$F(z)=f(rz)=U(z)+iV(z),~z\in\mathbb{D},$$ where $U(z)=u(rz)$ and  $V(z)=v(rz)$.
It follows from (\ref{fx-1})  that

\beqq
M_{p}(r,v)=\left(\frac{1}{2\pi}\int_{0}^{2\pi}|v(re^{i\theta})|^{p}d\theta\right)^{\frac{1}{p}}=\left(\frac{1}{2\pi}\int_{0}^{2\pi}|V(e^{i\theta})|^{p}d\theta\right)^{\frac{1}{p}}
=\|V\|_{p}
\eeqq
and
\beqq
M_{p}(r,u)=\left(\frac{1}{2\pi}\int_{0}^{2\pi}|u(re^{i\theta})|^{p}d\theta\right)^{\frac{1}{p}}=\left(\frac{1}{2\pi}\int_{0}^{2\pi}|U(e^{i\theta})|^{p}d\theta\right)^{\frac{1}{p}}
=\|U\|_{p},
\eeqq
which, together with (\ref{x-3}), imply that

\beqq
M_{p}(r,v)\leq2^{\frac{\max\{0,\frac{p}{2}-1\}}{p}+\frac{1}{2}}C\left( \frac{KC}{\sin\frac{\pi}{2p^{\ast}}}M_{p}(r,u)+\sqrt{K'}\right).
\eeqq
In particular, if  $K-1=K'=0$, then $C_{1}(1,p)=\cot(\pi/(2p^{\ast}))$ follows from (\ref{Riez-1}).
The proof of this theorem is complete.
 \qed

The following lemma is  useful for proving Theorem \ref{maint}. Since its proof is very elementary, we omit it here.

\begin{lem}\label{simple}
For $p\geq2$, we have $$|x+iy|^{p-2}\le c_p\left(|x|^{p-2}+|y|^{p-2}\right)$$
and $$|x|^p+|y|^p\le 2^{1 - p/2}|x+i y|^p,$$
 where $x,~y\in\mathbb{R}$ and $$c_p=\left\{
 \begin{array}{ll}
  1, & \hbox{for $2<p\le 4$} \\
  2^{\frac{1}{2} (-4+p)}, & \hbox{for $p\ge 4$.}
   \end{array}
   \right.$$
\end{lem}

The following lemma will play an important role in the proof of Theorem \ref{maint}.
\begin{lem}\label{vpge2}
Let $$\mathcal{A}_p=\left\{
                  \begin{array}{ll}
\left[-\pi, \pi\right], & \hbox{if $p$ is an even integer,} \\
\left[-\frac{\pi}{2}+\frac{\pi}{p}, \frac{3\pi}{2}-\frac{\pi}{p}\right], & \hbox{if $p\in[2,\infty)\setminus  2\mathbb{Z}$.} \\
                  \end{array}
                \right.$$
   \begin{enumerate}
\item[{\rm (I)}]   For any $z\in \mathbb{C}$ and $p\ge 2$,  \begin{equation}\label{ineq1}
|{\rm Im} (z)|^p\le a_p |{\rm Re} (z) |^p- b_p \Phi_p(z),
\end{equation}
and for $z=|z|e^{it}$, $t\in\mathcal{A}_p$,
 \begin{equation}\label{ineq2}
|{\rm Im} (z)|^p\le a_p |{\rm Re} (z) |^p+ b_p |z|^p \cos (p(\pi/2-t)),
\end{equation}
where $\Phi_p$  is defined in (\ref{rf-u-1}), $$a_p=\cot^p\left(\frac{\pi }{2 p}\right)~\mbox{and}~b_p=2 \cos^p\left(\frac{\pi }{2 p}\right) \csc\left(\frac{\pi }{p}\right).$$
\\
\item[{\rm (II)}] For any $z\in \mathbb{C}$ and $p\ge 2$,  \begin{equation}\label{ineq3}
| z|^p\le m_p |{\rm Re}(z) |^p- n_p \Phi_p(z),
\end{equation}
and for $z=|z|e^{it}$, $t\in\mathcal{A}_p$, \begin{equation}\label{ineq4}
| z|^p\le m_p |{\rm Re}(z) |^p+n_p |z|^p\cos p(\pi/2-t),
\end{equation}
where $$m_p=\csc^p\left(\frac{\pi }{2 p}\right)~\mbox{and}~ n_p= \cot\frac{\pi}{2p}.$$
\end{enumerate}
\end{lem}
\begin{proof} Let us make some reduction at start. Without loss of generality, we assume $z\neq0$.
For $z=|z|e^{it}\in\mathbb{C}$ with $t=\arg z$, let $$LHS_1(t)=|{\rm Im} (z)|^p,~RHS_1(t)=a_p |{\rm Re} (z) |^p- b_p \Phi_p(z),~LHS_3(t)=| z|^p$$
and $$RHS_3(t)=m_p |{\rm Re}(z) |^p- n_p \Phi_p(z)).$$
On the other hand, for $z=|z|e^{it}$ with $t\in\mathcal{A}_p$,
let $$LHS_2(t)=|{\rm Im} (z)|^p,~RHS_2(t)=a_p |{\rm Re} (z) |^p+ b_p |z|^p \cos (p(\pi/2-t)),~LHS_4(t)=| z|^p$$
and $$RHS_4(t)=m_p |{\rm Re}(z) |^p+n_p |z|^p\cos p(\pi/2-t).$$
 In all inequalities in \eqref{ineq1}-\eqref{ineq4}, for $j=1,2,3,4$, we have to prove that $\Psi_j(t)=LHS_j(t)-RHS_j(t)\le 0$. Since $\Psi_j(\pi-t)=\Psi_j(t)$,
 instead of proving the inequalities \eqref{ineq1}-\eqref{ineq4} for $[\pi/p-\pi/2,3\pi/2-\pi/p]$, we prove them  for $[\pi/p-\pi/2,\pi/2]$. For even $p$, $\Psi_j(-t)=\Psi_j(t)$. So the inequalities \eqref{ineq1}-\eqref{ineq4} will follow if we prove them for $[0,\pi]$.

We first prove {\rm (I)}. Assume first that $t\in[\pi/2-\pi/p,\pi/2]$. We need to prove that $$h(t):=\cot^p\frac{\pi }{2 p} \cos^p  t-2 \cos^p\frac{\pi }{2 p} \varphi(t)\csc   \left(\frac{\pi }{p}\right)-\sin^p  t\ge 0.$$ In this case $$\varphi(t)=-\cos  \left[p \left(\frac{\pi }{2}-t\right)\right].$$
Then, for $g(t) = h(t)\sin^{-p} t$, we have $$g'(t)=-\frac{p \left(\cos  ^p t \cot^p\frac{\pi }{2 p}+2 \cos^p\frac{\pi }{2 p} \cos  [t] \cos  \left[\frac{1}{2} p (\pi -2 t)+t\right] \csc\frac{\pi }{p}\right)}{ \cos t \sin^{1+p} t}.$$
We need to show that $t_\circ=\pi/2-\pi/(2p)$ is the only solution to $g'(t)=0$. This is equivalent to showing that $$k(t):=\cos^{p-1}  t\left/\cos  \left[ p (\pi -2 t)/2+t\right] \right.+2 \csc   \left(\frac{\pi }{p}\right) \sin^p  \left(\frac{\pi }{2 p}\right)$$ has only one zero in $[\pi/2-\pi/p,\pi/2]$. But
$$k'(t)=(1-p) \cos^{p-2}  t \sec^2\left[ p (\pi -2 t)/2+t\right] \sin  \left[ p (\pi -2 t)/2+2 t\right],$$ and this function has the constant sign for $t\in [\pi/2-\pi/p,\pi/2)$. This shows that $g$ has only one critical point in $[\pi/2-\pi/p,\pi/2)$. Since $$g''(\pi/2-\pi/(2p))=4 (-1+p) p \cot \left(\frac{\pi }{p}\right) \csc   \left(\frac{\pi }{p}\right)>0$$ and by \eqref{cos12} below $$g(\pi/2)=-1+2 \cos^p\frac{\pi }{2 p} \csc   \left(\frac{\pi }{p}\right)>0,$$ it follows that $t=\pi/2-\pi/(2p)$ is the local minimum of $g$, and moreover that minimum is equal to $g(\pi/2-\pi/(2p))=0$. Then also $g(\pi/2-\pi/p)>0$, otherwise $g$ will have one additional stationary point in $(\pi/2-\pi/p,\pi/2-\pi/(2p))$, which is impossible. This implies that $g(t)\ge 0$ for every $t\in [\pi/2-\pi/p,\pi/2]$ and so $h$ is non-negative in this interval. If $t\in[0, \pi/2-\pi/p]$, then there is $s\in[\pi/2-\pi/p, \pi/2]$ so that  $$\varphi(t)=\max\{|\cos p(\pi/2-t)|,|\cos p(\pi/2+t)|\}=-\cos(p(\pi/2-s)).$$ Then  \begin{equation}\label{bigsplit}\begin{split}h(t)&=\cot^p\frac{\pi }{2 p} \cos  ^p t+2 \cos^p\frac{\pi }{2 p} \varphi(t)\csc   \left(\frac{\pi }{p}\right)-\sin  ^p t\\&=\cot^p\frac{\pi }{2 p} \cos  ^p t+2 \cos^p\frac{\pi }{2 p} (\cos(p(\pi/2-s)))\csc   \left(\frac{\pi }{p}\right)-\sin  ^p t\\&\ge \cot^p\frac{\pi }{2 p} \cos^2  (s)+2 \cos^p\frac{\pi }{2 p} (\cos(p(\pi/2-s)))\csc   \left(\frac{\pi }{p}\right)-\sin^p s.
\end{split}\end{equation}
The last expression is non-negative by the previous case.
This implies \eqref{ineq1}.

To prove \eqref{ineq2}, observe that if $s\in[0,\pi/2-\pi/p]$, then there is $t\in [\pi/2-\pi/p,\pi/2]$, so that $\cos(p(\pi/2-s))=\varphi(t):=\cos(p(\pi/2-t))$, because $\varphi[\pi/2-\pi/p,\pi/2]=[-1,1]$. Then we argue as in \eqref{bigsplit} and obtain the desired conclusion. Finally, if $\frac{\pi}{p}-\frac{\pi}{2}\le t\le 0$, then there is $s\in [\pi/2-\pi/p,\pi/2]$, so that $\cos(p(\pi/2-t))=\cos(p(\pi/2-s))$ and again use the previous argument. Observe that $|s|\le |t|$ and in \eqref{bigsplit} instead of $\sin s$ we should have $|\sin s|$.

Next, we prove {\rm (II)}.  Similarly as before we only prove the inequalities for $t\in[\pi/2-\pi/p,\pi/2]$. In this case we need to prove the inequality $$1-\cos  \left[p \left(\frac{\pi }{2}-t\right)\right] \cot\frac{\pi }{2 p}-\cos^p  t \csc^p   \left(\frac{\pi }{2 p}\right)\le 0$$ for $t\in[\pi/2-\pi/p,\pi/2]$. Then we make the substitution $t=\pi/2-t$ and arrive to the equivalent inequality
$$\psi(s):=1-\cos  (ps) \cot\frac{\pi }{2 p}-\csc^p   \left(\frac{\pi }{2 p}\right) \sin^p s $$ for $s\in[0,\pi/p]$.
Further for $\phi(s)=\sin^{-p}(s)\psi (s)$ we have $$\phi'(s) = p \sin^{-1-p}  (s) \left[\cos  (s) \left(-1+\cos  (ps) \cot\frac{\pi }{2 p}\right)+\cot\frac{\pi }{2 p} \sin  (s) \sin  (ps)\right].$$
Let $$\zeta(s)= \cos  (s) \left(-1+\cos  (ps) \cot\frac{\pi }{2 p}\right)+\cot\frac{\pi }{2 p} \sin  (s) \sin  (ps).$$
Then $$\zeta'(s) = \sin  (s)-(p-1)\cot\frac{\pi }{2 p} \sin  [(p-1)s].$$ We prove now that $\zeta'(s)<0$. Let $q=p-1$ and define $\sigma(s) =\sin(qs)/(\sin{s})$. Prove that $\sigma(s)$ is monotonous decreasing.
We have $$\sigma'(s)= \csc   (s) \sin  (q s) (-\cot (s)+q \cot (q s)).$$ Further $$ \partial_s (-\cot (s)+q \cot (q s))=\csc^2 s -q^2 \csc^2(q s) .$$ The last expression is negative for $s\in(0,\pi/(q+1)]$ if and only if $\sin  (q s)-q \sin  (s)\le 0$. Finally $$\partial_s (\sin  (q s)-q \sin  (s))=q( \cos (q s) -  \cos ( s))\le 0$$ if $s\in (0,\pi/(q+1)]$, because $s<qs < \pi$ and $\cos$ is monotonous in $[0,\pi]$. Thus $$\sin  (q s)-q \sin  (s)\le \sin  (q \cdot 0)-q \sin0=0.$$ Therefore $$-\cot (s)+q \cot (q s)\le \lim_{s\to 0} \left(-\cot s+q \cot (q s)\right)=0.$$ So $\sigma'(s)<0$ and thus $\sigma$ is decreasing.

Thus \[\begin{split}\frac{\zeta'(s)}{\sin s} &=(1-(p-1)\cot\frac{\pi }{2 p} \sigma(s))\\&\le (1-(p-1)\cot\frac{\pi }{2 p} \sigma(\pi/(p)))\\&=\chi(p):=1-(p-1) \cot\frac{\pi }{2 p}.\end{split}\]
Then $$\chi'(p)=\frac{(-1+p) \pi +p^2 \sin  \left(\frac{\pi }{p}\right)}{p^2 \left(-1+\cos  \left(\frac{\pi }{p}\right)\right)},$$ which is clearly negative. Thus
$$\frac{\zeta'(s)}{\sin s}\le \chi(p)\le \chi(2)=0.$$  Thus $\zeta$ is decreasing. Since $\zeta(\pi/(2p))=0$, we see that $\pi/(2p)$ is the only stationary point of $\phi$ for $s\in(0,\pi/p]$. Further $\lim_{s\to 0} \phi(s) =-\infty $ and $$\phi(\pi/p)=2^{-p}\tan^{p}\left(\frac{\pi}{2p}\right)  \left(-2^p \cos^p\frac{\pi }{2 p}+ \left(1+\cot\frac{\pi }{2 p}\right)\right).$$

We need to show that $\Phi(\pi/p)<0$.
First of all $$\partial_p \cos^p\frac{\pi }{2 p}=\cos^p\frac{\pi }{2 p} \omega(p),$$ where $$\omega(p)=\left(\log \left[\cos\frac{\pi }{2 p}\right]+\frac{\pi  \tan \left(\frac{\pi }{2 p}\right)}{2 p}\right).$$

Since $$\omega'(p)=-\frac{\pi ^2 \sec^2\left(\frac{\pi }{2 p}\right)}{4 p^3}<0$$ we obtain that $$\partial_p \cos^p\frac{\pi }{2 p}\ge \lim_{p\to \infty}\cos^p\frac{\pi }{2 p} \omega(p)=0.$$ Thus \begin{equation}\label{cos12}\cos^p\frac{\pi }{2 p}\ge \frac{1}{2}.\end{equation}
So, to prove $\phi(\pi/2)<0$, it is enough to show that for $p\ge 2$ $$1-2^{-1+p}+\cot\frac{\pi }{2 p}\le 0.$$

But $$1-2^{-1+p}+\cot\frac{\pi }{2 p}\le \beta(p):=1-p+\cot\frac{\pi }{2 p}. $$ Finally $$\beta'(p)=\csc^2 x \left(\frac{2 x^2}{\pi }-\sin  ^2 x\right),$$ where $x=\frac{\pi}{2p}\in[0,\pi/4]$.  But, $\sin x$ is concave in $[0,\pi/4]$ and thus $\sin x>\frac{2 \sqrt{2} x}{\pi }$ for $x\in[0,\pi/4]$. Thus $\beta'(p)\le 0$. In particular $\beta(p)\le \beta (2)=0$ for every $p\ge 2$. Thus $\phi$ is non-positive  in $[0,\pi/p]$ and has its zero (maximum) at  $\pi/(2p)$. This finishes the proof of the case $t\in[\pi/2-\pi/p,\pi/2]$. Similarly as in (I) we reduce the case  $t\in [0,\pi/p]$ to the former case.
\end{proof}

\subsection{The proof of Theorem \ref{maint}}
Let  $f=h+\overline{g}$ be a harmonic $K$-quasiregular mapping of $\mathbb{D}$ into $\mathbb{C}\setminus(-\infty, 0]$, where $h$ and $g$ are analytic functions in $\mathbb{D}$.
Then for $p\in[2,\infty)$, $F_p={\rm Re}(-(-if)^p)$ is well-defined and smooth in $\mathbb{D}$. In particular, if $p\geq2$ is an integer, then $F_p$ is smooth in $\mathbb{D}$ provided that $f$ is an arbitrary harmonic $K$-quasiregular mapping in $\mathbb{D}$.
By straightforward calculations, we get
\beq\label{fut}
|\Delta F_p|&\le& p(p-1) R^{p-2}|f_x^2+f_y^2|=4p(p-1)R^{p-2}|g'h'|\\ \nonumber
&\le& 4 k p(p-1) R^{p-2}| h'|^2,
\eeq where $R=|f|$ and $k=(K-1)/(K+1)$.
Now we use \cite[Theorem~9.9]{gt} to conclude that

\be\label{Green}F_p(0)=\frac{1}{2\pi}\int_{-\pi}^{\pi}F_p(e^{it})dt
 -\frac{1}{2}\int_{\mathbb{D}}\Delta (F_p(z))\log \frac{1}{|z|}dA(z),\ee
where $dA(z)=dxdy/\pi$ denotes the normalized area measure in
$\mathbb{D}$.

 By Lemma~\ref{vpge2}, we have

 $$|f(r e^{it})|^p\le m_p|u(r e^{it})|^p-n_pF_p(r e^{it}).$$ Therefore \begin{equation}\label{vf2}\|f\|_p^p\le m_p\|u\|_p^p-n_p\int_{-\pi}^{\pi}F_p( e^{it})\frac{dt}{2\pi}.\end{equation}
Moreover
$$|v(r e^{it})|^p\le a_p|u(r e^{it})|^p-b_pF_p(r e^{it}).$$ Hence \begin{equation}\label{vf}\|v\|_p^p\le a_p\|u\|_p^p-b_p\int_{-\pi}^{\pi}F_p( e^{it})\frac{dt}{2\pi}.\end{equation}
Further, by \cite[Theorem~9.9]{gt},  we have
$$\int_{-\pi}^{\pi}|u( e^{it})|^p\frac{dt}{2\pi}= |u(0)|^p+\frac{1}{2}\int_{\mathbb{D}}\Delta (|u(z)|^p) \log\frac{1}{|z|}dA(z).$$
Further, because $u={\rm Re} (g+h)$, we have
\begin{equation}\label{equ}\begin{split}\Delta (|u|^p)&=p(p-1)|\nabla u|^2|u|^{p-2}\\&=p(p-1)|g'+h'|^2|u|^{p-2}\\&\ge p(p-1)(1-k)^2|h'|^2 |u|^{p-2}.\end{split}\end{equation}
Similarly, again using \cite[Theorem~9.9]{gt}, we have $$\int_{-\pi}^{\pi}|v( e^{it})|^p\frac{dt}{2\pi}= |v(0)|^p+\frac{1}{2}\int_{\mathbb{D}}\Delta (|v(w)|^p) \log\frac{1}{|w|}dA(w)$$ and

\begin{equation}\label{equ1}\begin{split}\Delta (|v|^p)&=p(p-1)|\nabla v|^2|v|^{p-2}\\&=p(p-1)|g'-h'|^2|v|^{p-2}\\&\ge p(p-1)(1-k)^2|h'|^2 |v|^{p-2}.\end{split}\end{equation}

Then \eqref{equ} and \eqref{equ1} and the first equation in Lemma~\ref{simple} imply that
 \begin{equation}\label{futja}
 \begin{split}|\Delta F_p(z)|&\le 4 p(p-1) R^{p-2}|g'\cdot h'|\\&\le 4 k p(p-1) c_p\left(|u|^{p-2}+|v|^{p-2}\right)| h'|^2\\ \nonumber
&\leq\frac{4 k  c_p}{(1-k)^2}(\Delta (|u|^p)+\Delta (|v|^p))\\&=(K^2-1)c_p(\Delta (|u|^p)+\Delta (|v|^p)). \end{split}\end{equation}
Thus
\beq\label{LKJ-1}\int_{-\pi}^{\pi}F_p( e^{it})\frac{dt}{2\pi}&=&
F_p(0)+\frac{1}{2}\int_{\mathbb{D}}\Delta (F_p(w)) \log\frac{1}{|w|}dA(W),
\\ \nonumber
& \ge&  F_p(0)-\frac{1}{2\pi}\int_{\mathbb{D}}|\Delta (F_p(w))| \log\frac{1}{|w|}dA(w)
\\ \nonumber
& \ge& F_p(0)-(K^2-1)c_p\int_{\mathbb{D}}\Delta \left(|u(w)|^p+|v(w)|^p\right)\log\frac{1}{|w|}dA(w).
\eeq
By \eqref{eqcon}, we have $F_p(0)\ge 0$, which, together with (\ref{LKJ-1}), gives that
\[\begin{split}&\int_{-\pi}^{\pi}F_p(e^{it})\frac{dt}{2\pi} \ge -(K^2-1)c_p(\|u\|_p^p+\|v\|_p^p).\end{split}\] Now  by \eqref{vf2} and second relation in Lemma~\ref{simple}, we have \[\begin{split}\|f\|_p^p&\le m_p\|u\|_p^p\\&-n_p\left((K^2-1)c_p(|u(0)|^p+|v(0)|^p)-(K^2-1)c_p(\|u\|_p^p+\|v\|_p^p)\right)\\ &\le m_p\|u\|_p^p+ (K^2-1)c_p n_p 2^{1-p/2} \|f\|_p^p.
\end{split}\]
Thus, for $(K^2-1)c_p n_p 2^{1-p/2}<1$, we obtian
 $$\|f\|_p\le C(K,p)\|u\|_p,$$ where $$C^p(K,p)=m_p\left(1-(K^2-1)c_p n_p 2^{1-p/2}\right)^{-1}.$$
 Observe that $C^p(1,p)=m_p$ is sharp by the result of Verbitsky.

Further, by \eqref{vf}, we have $$(1-(K^2-1)b_pc_p)\|v\|_p^p\le (a_p+(K^2-1)b_pc_p)\|u\|_p^p.$$
We assume that $(K^2-1)b_pc_p<1$. Then
 $$\| v\|_p\le c(K,p)\|u\|_p,$$ where

$$c(K,p)=\frac{(a_p+(K^2-1)b_pc_p)^{\frac{1}{p}}}{(1-(K^2-1)b_pc_p)^{\frac{1}{p}}}$$ and again for $K=1$, $c^p(1,p)=a_p$ is sharp by Pichorides result. Notice that both constants $c(K,p)$ and $C(K,p)$ are meaningful for small constant $K$. On the other hand Theorem \ref{thm-0.1} guarantees the existence of constant for big $K$.
\qed



\begin{Thm}{\rm (\cite[Theorem  2]{H-L})}\label{Thm-HL-2}
Let $f=u+iv$ be holomorphic in $\mathbb{D}$. For $p\in(0,1]$ and $\iota\in[0,\infty)$, if $$M_{p}(r,u)=O\left(\frac{1}{(1-r)^{\iota}}\right)$$ as $r\rightarrow1^{-}$,
then $$M_{p}(r,f')=O\left(\frac{1}{(1-r)^{1+\iota}}\right)$$ as $r\rightarrow1^{-}$.
\end{Thm}

\subsection*{The proof of Theorem \ref{thm-K-L}}
 Similar to the proof of Theorem \ref{thm-0.1}, we let $f=h + \overline{g}$, where $h$ and $g$ are analytic
 in $\mathbb{D}$ with $g(0)=0$. Let $F_{1}=h+g$ and $F_{2}=h-g$. Then $u={\rm Re}(F_{1})$ and $v={\rm Im}(F_{2})$.

 We first prove $(\mathscr{B}_{1})$.
It follows from $u\in\mathbf{h}^{1}$ and Theorem D that ${\rm Im}(F_{1})\in\mathbf{h}^{q}$ for all $q\in(0,1)$.
This together with Lemma I and H\"older's inequality give
\beqq
M_{q}^{q}(r,F_{1})&\leq&M_{q}^{q}(r,u)+M_{q}^{q}(r,{\rm Im}(F_{1}))
\leq\,M_{q}^{q}(r,u)+\|{\rm Im}(F_{1})\|_{q}^{q}\\
&\leq&M_{1}^{q}(r,u)\left(\frac{1}{2\pi}\int_{0}^{2\pi}1^{\frac{1}{1-q}}d\theta\right)^{1-q}+\|{\rm Im}(F_{1})\|_{q}^{q}\\
&\leq&\|u\|_{1}^{q}+\|{\rm Im}(F_{1})\|_{q}^{q}<\infty.
\eeqq Then $F_{1}\in\mathbf{h}^{q}$.
By Theorem H, we see that there is a positive constant $C=C(q)$  such that
\beq\label{eq-ghk-1}
\frac{1}{2\pi}\int_{0}^{2\pi}\big(\mathscr{G}[F_{1}](e^{i\theta})\big)^{q}d\theta&=&
\frac{1}{2\pi}\int_{0}^{2\pi}\left(\int_{0}^{1}(1-r)|F_{1}'(re^{i\theta})|^{2}dr\right)^{\frac{q}{2}}d\theta\\ \nonumber
&\leq&C^{q}\|F_{1}\|_{q}^{q}.
\eeq




Combing (\ref{eq-ghk-2}), (\ref{eq-ghk-1})  and Lemma I gives
\beqq
\|\mathscr{G}[F_{2}]\|_{L^{q}}&=&\left(\frac{1}{2\pi}\int_{0}^{2\pi}\left(\int_{0}^{1}(1-r)|F_{2}'(re^{i\theta})|^{2}dr\right)^{\frac{q}{2}}d\theta\right)^{\frac{1}{q}}\\
&\leq&\left(\frac{1}{2\pi}\int_{0}^{2\pi}\left(\int_{0}^{1}(1-r)\left(2K^{2}|F_{1}'(re^{i\theta})|^{2}+2K'\right)dr\right)^{\frac{q}{2}}d\theta\right)^{\frac{1}{q}}\\
&\leq&2^{\frac{1}{q}-\frac{1}{2}}\left(K\|\mathscr{G}[F_{1}]\|_{L^{q}}+\sqrt{K'}\right).
 \eeqq
 This together with  Theorem H yields that
 \beqq
 \|F_{2}\|_{q}&\leq&C\|\mathscr{G}[F_{2}]\|_{L^{p}}\leq2^{\frac{1}{q}-\frac{1}{2}}C\left(K\|\mathscr{G}[F_{1}]\|_{L^{q}}+\sqrt{K'}\right)\\
 &\leq&2^{\frac{1}{q}-\frac{1}{2}}C\left(KC\|F_{1}\|_{q}+\sqrt{K'}\right)
 \eeqq
 where $C=C(q)$ is a positive constant.
Consequently, \be\label{K-03}\|v\|_{q}\leq\|F_{2}\|_{q}\leq\,2^{\frac{1}{q}-\frac{1}{2}}C\left(KC\|F_{1}\|_{q}+\sqrt{K'}\right)<\infty.\ee

Next, we prove  $(\mathscr{B}_{2})$. By the assumption and Theorem E, we see that ${\rm Im}(F_{1})\in\mathbf{h}^{1}$.
It follows from  Theorem H that there is a positive constant $C$, independently of $f$,  such that
\beq\label{eq-ghk-5}
\frac{1}{2\pi}\int_{0}^{2\pi}\mathscr{G}[F_{1}](e^{i\theta})d\theta&=&
\frac{1}{2\pi}\int_{0}^{2\pi}\left(\int_{0}^{1}(1-r)|F_{1}'(re^{i\theta})|^{2}dr\right)^{\frac{1}{2}}d\theta\\ \nonumber
&\leq&C\|F_{1}\|_{1}<\infty.
\eeq
Combing   (\ref{K-03}) and (\ref{eq-ghk-5}) yields
\beqq
\|v\|_{1}\leq\|F_{2}\|_{1}\leq\,2^{\frac{1}{2}}C\left(KC\|F_{1}\|_{1}+\sqrt{K'}\right)<\infty.
 \eeqq

Now, we prove $(\mathscr{B}_{3})$. Since $u\in\mathbf{h}^{p}$ for some $p\in(0,1)$, we see that
there is positive constant $C$ such that
\beqq
M_{p}(r,u)\leq\,C.
\eeqq
By Theorem O, there is positive constant $C$ such that
$$M_{p}(r,F_{1}')\leq\frac{C}{(1-r)},$$
which, together with (\ref{eq-ghk-2}), gives that
\be\label{eq-hn-1}M_{p}(r,F_{2}')\leq2^{\frac{1}{p}-1}\left(KM_{p}(r,F_{1}')+\sqrt{K'}\right)\leq\frac{C}{(1-r)}.\ee
It follows from (\ref{eq-hn-1}) and Theorem H  that
$$M_{p}(r,F_{2})=O\left(\left(\log\frac{1}{1-r}\right)^{\frac{1}{p}}\right)$$ as $r\rightarrow1^{-}$,
which implies that
$$M_{p}(r,v)=O\left(\left(\log\frac{1}{1-r}\right)^{\frac{1}{p}}\right)$$ as $r\rightarrow1^{-}$.

At last, we prove the sharpness part of (\ref{jk-1q}).
 Consider the functions
$$f_{j}(z)=\frac{2K}{K+1}{\rm Re}(h_{j}(z))+i\frac{2}{K+1}{\rm Im}(h_{j}(z)),~z\in\mathbb{D},$$
where $h_{j}(z)=e^{\frac{j\pi}{2}i}(1-z)^{-j-1}$ for $j\in\{0,1,\ldots\}$.
Then ${\rm Re}(f_{j})\in\mathbf{h}^{\frac{1}{j+1}}$ and
\beq\label{eq-ngk-10}
\frac{1}{2\pi}\int_{0}^{2\pi}|{\rm Im}(f_{j}(re^{i\theta}))|^{\frac{1}{1+j}}d\theta&=&
\frac{1}{2\pi}\left(\frac{2}{K+1}\right)^{\frac{1}{1+j}}\int_{0}^{2\pi}|{\rm Im}(h_{j}(re^{i\theta}))|^{\frac{1}{1+j}}d\theta\\ \nonumber
&\leq&\frac{1}{2\pi}\left(\frac{2}{K+1}\right)^{\frac{1}{1+j}}\int_{0}^{2\pi}|h_{j}(re^{i\theta})|^{\frac{1}{1+j}}d\theta\\ \nonumber
&=&O\left(\log\frac{1}{1-r}\right)
\eeq as $r\rightarrow1^{-}$. On the other hand,

\beqq
\frac{1}{2\pi}\int_{0}^{2\pi}|h_{j}(re^{i\theta})|^{\frac{1}{1+j}}d\theta&\leq&\left(\frac{K+1}{2}\right)^{\frac{1}{1+j}}
\frac{1}{2\pi}\int_{0}^{2\pi}|f_{j}(re^{i\theta})|^{\frac{1}{1+j}}d\theta\\
&\leq&\left(\frac{K+1}{2}\right)^{\frac{1}{1+j}}\left(\|{\rm Re}(f_{j})\|_{1/(1+j)}^{1/(1+j)}+
\int_{0}^{2\pi}|{\rm Im}(f_{j}(re^{i\theta}))|^{\frac{1}{1+j}}\frac{d\theta}{2\pi}\right),
\eeqq
which implies that
\beq\label{eq-ngk-11}
\frac{1}{2\pi}\int_{0}^{2\pi}|{\rm Im}(f_{j}(re^{i\theta}))|^{\frac{1}{1+j}}d\theta&\geq&
\left(\frac{2}{K+1}\right)^{\frac{1}{j+1}}\frac{1}{2\pi}\int_{0}^{2\pi}|h_{j}(re^{i\theta})|^{\frac{1}{1+j}}d\theta\\ \nonumber
&&-\|{\rm Re}(f_{j})\|_{1/(1+j)}^{1/(1+j)}\\ \nonumber
&=&O\left(\log\frac{1}{1-r}\right)
\eeq as $r\rightarrow1^{-}$.
Combing (\ref{eq-ngk-10}) and  (\ref{eq-ngk-11}) implies the estimate of {\rm (\ref{jk-1q})} is sharp for
$p=1,\,1/2,\,1/3,\,\ldots$.
The proof of this theorem is finished.
\qed

\subsection*{The proof of Theorem \ref{thm-07}}
Similar to the proof of Theorem \ref{thm-0.1}, we let $f=h + \overline{g}$, where $h$ and $g$ are analytic
 in $\mathbb{D}$ with $g(0)=0$. Let $F_{1}=h+g$ and $F_{2}=h-g$. Then $u={\rm Re}(F_{1})$ and $v={\rm Im}(F_{2})$.
 Since $|F_{1}'|=|\nabla u|\in\mathbf{H}^{p}_{g}$, by Theorem G, we see that

\be\label{eq-l-01}F_{1}\in H^{q}.\ee
It follows from Theorem H that there is a positive constant $C=C(q)$ such that

\beqq
\|\mathscr{G}[F_{1}]\|_{L^{q}}\leq\,C\|F_{1}\|_{q},
\eeqq
which, together with  (\ref{eq-ghk-2}) and Lemma I, gives
\beq\label{eq-l-02}
\|\mathscr{G}[F_{2}]\|_{L^{q}}
&\leq&\left(\frac{1}{2\pi}\int_{0}^{2\pi}\left(\int_{0}^{1}(1-r)\left(2K^{2}|F_{1}'(re^{i\theta})|^{2}+2K'\right)dr\right)^{\frac{q}{2}}d\theta\right)^{\frac{1}{q}}\\ \nonumber
&\leq&2^{\frac{1}{q}-\frac{1}{2}}\left(K\|\mathscr{G}[F_{1}]\|_{L^{q}}+\sqrt{K'}\right)\\ \nonumber
&\leq&2^{\frac{1}{q}-\frac{1}{2}}\left(KC\|F_{1}\|_{q}+\sqrt{K'}\right)\\  \nonumber
&<&\infty.
 \eeq
By (\ref{eq-l-02}) and Theorem H, there is a positive constant $C=C(q)$ such that

\be\label{eq-l-03}
\|F_{2}\|_{q}\leq\,C\|\mathscr{G}[F_{2}]\|_{L^{q}}<\infty.
\ee
Since $$f=\frac{1}{2}(F_{1}+F_{2})+\frac{1}{2}(\overline{F_{1}}-\overline{F_{2}}),$$
we see that
\be\label{eq-l-04}
|f|^{q}\leq(|F_{1}|+|F_{2}|)^{q}\leq|F_{1}|^{q}+|F_{2}|^{q}.
\ee
Combining (\ref{eq-l-01}), (\ref{eq-l-03}) and (\ref{eq-l-04}) gives
\beqq
\|f\|_{q}^{q}\leq\|F_{1}\|_{q}^{q}+\|F_{2}\|_{q}^{q}<\infty.
\eeqq

To show that $q=p/(1-p)$ can not be replaced by any larger index, we consider the function

\beqq
f(z)=\frac{2K}{K+1}{\rm Re}(\Phi(z))+i\frac{2}{K+1}{\rm Im}(\Phi(z)),~z\in\mathbb{D},
\eeqq where $\Phi'(z)=1/(1-z)^{\epsilon-1/p}$ for some small $\epsilon>0$.
The proof of this theorem is complete.
 \qed

\bigskip

{\bf Data Availability} Our manuscript has no associated data.\\

{\bf Conflict of interest} The authors declare that they have no conflict of interest.

\bigskip

{\bf Acknowledgements.}\label{acknowledgments}
 The research of the first author was partly supported by the National Science
Foundation of China (grant no. 12571080).

\normalsize

\end{document}